\documentclass[graybox]{svmult}
\usepackage{mathptmx}       
\usepackage{helvet}         
\usepackage{courier}        
\usepackage{makeidx}         
\usepackage{graphicx}        
\usepackage{multicol}        
\usepackage[bottom]{footmisc}
\usepackage{amsmath,amssymb,amsfonts}        
\makeindex             

\usepackage{amssymb}
\usepackage{amsmath,amscd,amsxtra}
\usepackage{stmaryrd}
\usepackage{mathrsfs}
\usepackage{verbatim}

\usepackage[hidelinks]{hyperref}
\usepackage{latexsym}
\usepackage{enumerate}
\usepackage[boxed,vlined,ruled,linesnumbered,algochapter]{algorithm2e}
\usepackage{epigraph}
\usepackage{makeidx}
\usepackage[toc,page]{appendix}
\usepackage{color}
\usepackage{mathdots}
\usepackage{rotating}
\usepackage{subcaption}
\usepackage{bibentry}

 \newcommand{\be}{\begin{equation}}
\newcommand{\ee}{\end{equation}}
\newcommand{\bes}{\begin{equation*}}
\newcommand{\ees}{\end{equation*}}

\newcommand{\rd}{\mathrm{d}}
\newcommand{\ba}{\begin{array}}
\newcommand{\ea}{\end{array}}
\newcommand{\bg}{\begin{gathered}}
\newcommand{\eg}{\end{gathered}}
\def\bal#1\eal{\begin{align}#1\end{align}}
\def\bals#1\eals{\begin{align*}#1\end{align*}}

\newcommand{\bebg}{\begin{equation} \begin{gathered}}
\newcommand{\egee}{\end{gathered} \end{equation}}
\newcommand{\besbg}{\begin{equation*} \begin{gathered}}
\newcommand{\egees}{\end{gathered} \end{equation*}}

\newcommand{\beba}{\begin{equation} \begin{array}}
\newcommand{\eaee}{ \end{array} \end{equation}}
\newcommand{\besba}{\begin{equation*} \begin{array}}
\newcommand{\eaees}{\end{array} \end{equation*}}
\newcommand{\bn}{\mathbf{n}}
\newcommand{\bt}{\begin{Theorem}}
\newcommand{\et}{\end{Theorem}}
\newcommand{\bT}{\begin{theorem}}
\newcommand{\bTtext}[1]{\begin{Theorem}[#1] $ $}
\newcommand{\eT}{\end{theorem}}
\newcommand{\bL}{\begin{lemma}}
\newcommand{\bLs}{\begin{Lemma*} }
\newcommand{\bLtext}[1]{\begin{Lemma}[#1] $ $}
\newcommand{\eL}{\end{lemma}}
\newcommand{\eLs}{\end{Lemma*}}
\newcommand{\bC}{\begin{Corollary}}
\newcommand{\eC}{\end{Corollary}}
\newcommand{\bp}{\begin{proof} $ $ \newline }
\newcommand{\bptext}[1]{\begin{proof} #1 \newline }
\newcommand{\ep}{\end{proof}}
\newcommand{\bDtext}[1]{\begin{Definition}[#1] }
\newcommand{\eD}{\end{Definition}}
\newcommand{\bN}{\begin{Note}}
\newcommand{\bNs}{\begin{Note*}}
\newcommand{\eN}{\end{Note}}
\newcommand{\eNs}{\end{Note*}}

\newcommand{\br}{{\bf r}}

\newcommand{\bz}{{\bf z}}

\newcommand{\err}{{e}}

\newcommand{\sigT}{{\sigma}}
\newcommand{\sigS}{{\sigma_S}}
\newcommand{\sigA}{{\sigma_A}}

\newcommand{\sigThalf}{{\sigma^{1/2}}}

\newcommand{\Q}{Q}

\newcommand{\mcK}{\mathcal{K}}
\newcommand{\mcL}{\mathcal{L}}
\newcommand{\mcM}{\mathcal{M}}
\newcommand{\mcN}{\mathcal{N}}

\newcommand{\mcP}{\mathcal{P}}

\newcommand{\mcT}{\mathcal{T}}

\newcommand{\mcLhalf}{\mathcal{L}^{1/2}}

\newcommand{\Stwo}{{\mathbb{S}^2}}

\usepackage{soul}

\newcommand{\dint}[4]{\displaystyle{\int^{#1}_{#2} #3 \; \textrm{d}#4}}
\newcommand{\linedint}[4]{\displaystyle{\int^{#1}_{#2} #3 \; \textrm{d}l(#4)}} 

\newcommand{\norm}[1]{\left\| #1 \right\|}

\newcommand{\bexp}[1]{\exp{\left( #1 \right)}}

\newcounter{algcounter}

\begin{document} 

\title*{The radiative transport equation with heterogeneous cross-sections  }
\author{J.C.H. Blake, I.G. Graham, F. Scheben  and A. Spence }
\institute{J.C.H. Blake\at IBM, Hursley Park,  Winchester SO21 2JN, UK  \email{jackcheethamblake@gmail.com}
\and I.G. Graham  \at  Mathematical Sciences, University of Bath, BA2 7AY, UK  \email{I.G.Graham@bath.ac.uk} 
\and F. Scheben \at  Am Vogelbusch 25, 24568 Kattendorf, Germany 
   \email{F.Scheben@gmail.com} 
\and A. Spence   \at  Mathematical Sciences, University of Bath, BA2 7AY, UK  \email{A.Spence@bath.ac.uk}}
%
%

\titlerunning{Heterogeneous radiative transport equation}
\maketitle

\abstract{We consider the classical integral equation reformulation of the radiative transport equation (RTE) in
a   heterogeneous  medium, assuming isotropic scattering. We prove an estimate for the norm of the integral operator in this formulation which is explicit in the (variable) coefficients of the problem (also known as the cross-sections). This result uses only elementary properties of the transport operator and
  some classical functional analysis. As a corollary,  we  obtain a bound on the
  convergence rate of source iteration (a classical stationary iterative method for solving the RTE). We also obtain
  an estimate for the solution of the RTE which is explicit in its dependence on
  the cross-sections. The latter can be used to estimate the  solution in certain Bochner norms 
  when the cross-sections are random fields.
  Finally we use our results to give an elementary proof  that the generalised eigenvalue problem arising in
  nuclear reactor safety has only real and positive eigenvalues.    
}

\section{Introduction}
\label{sec:RTE}
In this note we present some elementary estimates  for the 
steady-state mono-energetic Radiative (Boltzmann or Neutron) Transport Equation (RTE) with
heterogeneous cross-sections.  Although these are relatively straightforward to prove, and there is a huge
literature on this topic, we were unable to
find proofs of these precise results in the literature and so it seems useful to record them here.
We emphasise that there are many excellent classical references for the general field discussed here - for example \cite{BelGla:70, LeMi:84, DaLi:12}. 
The  estimates given here  provide the tools needed to estimate the solution of the RTE explicitly in terms of the data (the so-called `cross-sections'). They also allow an explicit estimate for the rate of convergence of source iteration in terms
of the scattering ratio.  These estimates have recently proved essential for the rigorous analysis of uncertainty
quantification techniques  for the RTE \cite{GrPaSc:18, GrPaSc:18a, Pa:18}.   Most of the estimates presented here appeared   in the University of Bath 
PhD theses of Fynn Scheben and Jack Blake \cite{Sch:11, Bl:16, GrSc:11}. An  application to problems with random
data appears in the more recent University of  Bath  thesis of Matthew Parkinson  \cite{Pa:18} - see also \cite{GrPaSc:18, GrPaSc:18a}.

For $\br \in V\subset \mathbb{R}^3$,  where $V$ is 
a bounded convex spatial domain   and    $\Omega \in \Stwo$, the unit sphere in 3D, and     
 assuming   isotropic scattering, the RTE    source problem takes the form
\be
\label{eq:3DTEfull}
	\Omega\cdot\nabla\psi(\br,\Omega) + \sigT(\br)\psi(\br,\Omega) \ = \ \dfrac{\sigS(\br)}{4\pi}\dint{}{\Stwo}{\psi(\br,\Omega')}{\Omega'} + \Q(\br), 
        \ee
        where  $\nabla$ denotes the gradient with respect to $\br$, $Q$ is the source  and $\sigT$,  $\sigS$ are, respectively,   the {\em total} and {\em scattering} cross sections that satisfy
\be
\label{eq:3DCrossSections}
\sigT(\br) = \sigS(\br) + \sigA(\br), \quad \forall \br \in V, 
\ee
where     $\sigA$ is the {\em absorption} cross section.  All cross-sections 
are  assumed to be pointwise bounded  above and below on $V$ by strictly positive constants, i.e.,  {for all} \   $\br \in V$, 
$$0 < (\sigma_S )_{\min} \leq \sigma_S(\br) \leq (\sigma_S )_{\max}, \quad  \ 0 < (\sigma_A )_{\min} \leq \sigma_A(\br) \leq (\sigma_A )_{\max}, $$
\text{and thus} $$  0 < (\sigma)_{\min} \leq \sigma(\br) \leq (\sigma)_{\max} \ ,     $$
which ensures that  $\Vert \sigS/\sigma \Vert_{L^\infty(V)} < 1$.

In \eqref{eq:3DTEfull}, the   {\em angular flux} $\psi$ is to be found, subject to given boundary conditions. Here we only consider
the {\em vacuum} boundary condition on the inflow boundary:  
\be
\label{eq:3DnonzeroBCs}
	\psi(\br,\Omega) = 0 , \quad \textrm{when} \quad \bn(\br)\cdot\Omega < 0, \quad \br \in \partial V,  
        \ee
        where $\bn$ denoted the ourward normal from $V$. 
        
Introducing the transport and averaging  operators:
\be
\label{eq:3DToperator}
	\mcT\psi(\br,\Omega) = \Omega\cdot\nabla\psi(\br,\Omega) + \sigT(\br)\psi(\br,\Omega) \quad \text{and} \quad 	\mcP(\cdot) = \dfrac{1}{4\pi}\dint{}{\Stwo}{(\cdot)}{\Omega}, 
        \ee
  \eqref{eq:3DTEfull} can be rewritten as  
\be
\label{eq:3DTEop}
	\mcT\psi(\br,\Omega) = \sigS(\br)\phi(\br) + \Q(\br),
        \ee
        where  $\phi = \mcP \psi$ is called the {\em scalar flux}.
        This  is to be solved,  subject to  \eqref{eq:3DnonzeroBCs},

        Throughout, $L^2(V)$  will denote the space of square integrable functions on $V$, with inner product denoted $\langle \cdot, \cdot\rangle$.
        For any uniformly positive and bounded weight function $w$ on $V$, we will use  
$L^2(V,w)$ to denote the space of functions $v$,  
for which $ \Vert v \Vert_{L^2(V, w)}^2   : = \int_V \vert v(\br)\vert^2 w(\br) \rd r < \infty$.  When $w \equiv 1$, this reduces to
the standard $L^2$ norm $\Vert v \Vert_{L^2(V)}^2 = \langle v, v\rangle$.  

Before studying \eqref{eq:3DTEfull}, \eqref{eq:3DnonzeroBCs}
it is useful to first consider the `pure transport problem' 
$\mcT  \psi = g$,  subject to boundary  condition   \eqref{eq:3DnonzeroBCs} 
and with source $g \in L^2(V)$.
If the scalar flux $\phi$ were known,  then \eqref{eq:3DTEop} would allow computation of  
$\psi(\br, \Omega)$
for all $\br$ and any fixed  $\Omega$ by  solving   a single transport equation with vacuum boundary condition.
Such `transport sweeps' are `easy' operations both in theory (the solution of the  transport problem can be written down using characteristics - see Lemma \ref{lem:1}),  and in numerical practice  (e.g.,  when  discontinuous Galerkin methods are used to discretise the transport operator, the resulting linear system  can usually be solved by a single  sweep through the elements).
This motivates the use of the `source iteration' for solving \eqref{eq:3DTEop}, which
computes a sequence of approximations to $\phi$,  starting with an initial guess $\phi^{0} $ and iterating by solving :
\begin{align} \label{eq:SI} \mcT \psi^{i+1}(\br, \Omega)  = \sigS \phi^{i}(\br) + Q(\br), \quad
\text{subject to} \ \eqref{eq:3DnonzeroBCs} \ \text{and}   \quad \phi^{i+1}: = \mcP \psi^{i+1}.  \end{align} 
The fact that this iteration is always well-defined and is
equivalent to a fixed point iteration for a certain self-adjoint weakly singular integral operator
is established in    the following lemma.

%

              \bL \label{lem:1} 
	Let $g \in L^2(V)$ and consider  the pure transport problem:  Solve 
	$$
		\mcT \psi(\br,\Omega) = g(\br),  \quad \br \in V,  \ \Omega \in \Stwo, 
	$$
	  for $\psi(\br, \Omega)$, 
        subject to  boundary condition \eqref{eq:3DnonzeroBCs}.
        This problem has a unique solution given by 
\begin{align} \psi(\br, \Omega) \ = \ \int_0^{d(\br, \Omega)} \exp(-\tau(\br, \br - s \Omega)) g (\br - s \Omega) \rd s,  
\label{char} \end{align} 
where 
	$	\tau(\br,\br^\prime)  =   \linedint{}{l(\br,\br^\prime)}{\sigT(\bz)}{\bz}\ $ 
        (the integral of $\sigma$
        along the  line  $l(\br,\br^\prime)$ joining  $\br$ and  $\br^\prime$) and
         $$d(\br , \Omega) = \inf\{ s > 0: \br-s\Omega \not \in V\}. $$

        Moreover, 
         the correponding scalar flux  $\phi : = \mcP \psi$ can be
        expressed as 
        	\be
	\label{eq:3D_scalar_flux}
		\phi(\br) : =  (\mcK g)(\br),   
	\ee
	where $\mcK$ is the integral operator defined by 
	\be
        (\mcK v)(\br) \ :=\  \dint{}{V}{k(\br,\br^\prime)v(\br^\prime)}{\br^\prime},
        \quad \text{with kernel} \quad 		k(\br,\br^\prime) \ := \ \dfrac{\bexp{-\tau(\br,\br^\prime)}}{4\pi \norm{\br - \br^\prime}_2^2} .         \ee

                 \eL

                 \begin{proof}
                   It is easy to show  (using the method of  characteristics)  that the formula \eqref{char} provides the unique solution to the pure transport problem.
                This is well-known in the neutron transport literature, e.g. \cite{BelGla:70, PrLa:10} and    
                   $\tau$    is called the {\em optical path length}.                     
Once \eqref{char} is established, the formula \eqref{eq:3D_scalar_flux} is obtained by applying $\mcP$ to each side of \eqref{char}
and rewriting the result using  spherical
polar coordinates.     
\end{proof}

\begin{remark} An interesting observation is that, although the original transport problem is far from being self-adjoint, the integral operator $\mcK$ is self-adjoint
  with a positive kernel (and in fact is a  positive definite operator, as we see below). Since $\mcK$ has a
  weakly singular kernel and the domain of integration is a bounded  Euclidean domain,   it is to be expected that the solution $\phi$ will have (weak) boundary singularities. This
  property has been analysed in the classical literature 
   - see, e.g.,    \cite{PiSc:83} and the references therein.
  \end{remark} 

  Returning to source iteration \eqref{eq:SI}, we have the following simple corollary:

  \begin{corollary} \label{cor:SI} The iterates $\phi^i$ and the errors $e^i := \phi - \phi^i$ satisfy the equations
    \begin{align} \phi^{i+1} \ &=\ \mcK \sigma_S \phi^i + \mcK Q \nonumber \\
      \text{and} \quad e^{i+1} \ &= \ \mcK \sigma_S e^i  \label{eq:errEqn}
      \end{align} 
    \end{corollary}
                 



Hence  iteration \eqref{eq:SI}
will converge if and only if there is a norm in which the operator $\mcK \sigS$ is a  contraction.
(Here we emphasise that $\mcK \sigS$ denotes the composition of the operator of multiplication by $\sigS$ with the integral operator $\mcK$.)
In Theorem \ref{thm:1}  we prove that this is the case in  a certain  weighted $L^2$ norm induced by the total
cross-section $\sigT$.  Then Corollary \ref{cor:1} provides the result on the convergence of source iteration. 

In fact Theorem \ref{thm:1} has several other ramifications.  Combining Lemma \ref{lem:1} with \eqref{eq:3DTEop}   we obtain that the scalar $\phi$ satisfies the second kind weakly singular integral equation
\begin{align}\label{eq:IE} 
  \phi - \mcK \sigS \phi \ =\ \mcK Q , 
\end{align}
and Theorem \ref{thm:1} (and the Banach lemma) then readily tells us that this equation has a unique solution and provides a bound on its  norm explicit in the cross-sections
(Corollary \ref{cor:2}).

One reason for providing these results in this paper is that their proofs are hard to locate in the literature.
Another
reason is that they have direct relevance to the modern theory of unertainty quantification for the transport
equation. When the cross-sections $\sigma_S$ and $\sigma$ are random fields, then both the error estimates
for numerical methods for computing $\phi$ and also 
the rates of  convergence of iterative methods for computing realisations of the scalar flux $\phi$ depend expilcitly on   the cross-sections through the theorems presented here. This dependence is used explicitly
in recent work on UQ for transport problems \cite{GrPaSc:18, Pa:18,GrPaSc:18a}.   



It is known that source iteration converges when solving the neutron transport equation with
constant cross sections - see, e.g.,  \cite[Chapter 4]{Sch:11}.  Ashby et. al \cite[Section 4]{AshBroDorHin:95} prove a similar result with spatially dependent cross sections for a special discrete case. This work motivated us to consider a general  proof in the heterogeneous case. The results here are for the underlying operator before discretization. Extension to general discretizations is a complicated question.   

We will present the theory for the full 3D case where $\br \in V \subset \mathbb{R}^3$ and $\Omega \in \mathbb{S}^2$
the unit sphere in 3D, but the result also applies to the 2D reduction where $\br \in V \subset \mathbb{R}^2$ and $\Omega \in \mathbb{S}^1$  and to the case when space and angle are one-dimensional (the so called slab geometry case). Details of the proof in this case are given in \cite{Bl:16}. 


\section{The main result}
\label{sec:PreRes}

Our main goal in this section will be to prove the following theorem.

\begin{theorem} \label{thm:1} 
                For any function $\sigma^*$ satisfying $0 < \sigma^*_{\mathrm{min}} \leq \sigma^*(\br) \leq \sigma^*_{\mathrm{max}}$ for all $\br \in \overline{\Omega}$, 
  \begin{align} \label{eq:main}  \norm{\mcK \sigma^* }_{L^2(V, \sigma)} \ \leq \ \left \Vert \frac{\sigma^*}{\sigma} \right \Vert_{L^{\infty}(V)}.\end{align} 
\label{thm:NormBound}
\end{theorem}
(The left hand side of the inequality in \eqref{eq:main} denotes the operator norm of 
$\mcK \sigma^*$ on the space  $L^2(V)$,  equipped with the weighted norm $\Vert \cdot \Vert_{L^2(V, \sigma)}$.)    




The proof depends on several lemmas.  In these it is useful to
introduce the operator $$\mcL : =  \sigThalf\mcK\sigThalf.$$

  \begin{lemma} \label{lem:2}  The operators  $\mcK$ and $
		\mcL
	$ are  compact, self-adjoint and positive definite on $L^2(V)$.

  \end{lemma}

	
  \begin{proof}
    First, $\mcK$ is compact on $L^2(V)$
    because it is a weakly singular operator of potential type, see \cite[p.332]{KanAki:82}. To see the positive definiteness, let $g$ be an arbitrary function in $L^2(V)$ and let $\psi$ be the solution of $\mcT \psi = g$, subject to vacuum boundary conditions \eqref{eq:3DnonzeroBCs}. Then
    \begin{align*} \psi(\br, \Omega) g(\br) \ &= \ \psi(\br, \Omega ) \Omega \cdot  \nabla \psi(\br, \Omega) + \sigma \psi^2(\br, \Omega)\\
                                              & = \ \frac12 \nabla \cdot  (\Omega \psi^2(\br, \Omega) ) + \sigma(\br) \psi^2(\br, \Omega) .  \end{align*}
                                            Hence, integrating over $V$ and using the divergence theorem, we obtain 
                                            \begin{align} \int_V \psi(\br, \Omega) g(\br) \rd r \ &= \ \frac12 \int_{\partial V} \psi^2(\br, \Omega ) \Omega. \bn(\br) \rd s   + \int_V \sigma(\br)  \psi^2(\br, \Omega)\rd r \nonumber \\
                                                                                                  & \geq  \int_V \sigma(\br) \psi^2(\br, \Omega),  \label{daggar} \end{align}
                                                                                                where we used the vacuum boundary condition to get the final inequality in \eqref{daggar}. 
                                            Now introducing $\phi = \mcP \psi$, and recalling  \eqref{eq:3D_scalar_flux},  we also have $\mcP\psi = \phi = \mcK g $. Applying $\mcP$ to each side of \eqref{daggar}, we then obtain
                                            \begin{align*}
\langle g, \mcK g \rangle \ = \ 
                                              \int_V \phi(\br)  g(\br) \rd r \ 
                                              & \geq  \ \frac{1}{4\pi} \int_V \sigma(\br) \int_{\mathbb{S}^2} \psi^2(\br, \Omega) \rd \Omega \rd \br \ .  \end{align*}
                           This proves the postitive definiteness of $\mcK$. Since $\mcL$ is a simple left and right scaling of $\mcK$ with the positive-valued  function $\sigma^{1/2}$, the proof for $\mcL$ follows directly.                     


\end{proof} 

Our next result concerns an  upper bound on   $\mcL$.

\bL
	\label{lem:RayQuotBound}
	
	\bes
		\langle g, {\mcL g}\rangle \  \leq\ \Vert g \Vert_{L^2(V)}^2  , \quad \text{for all} \quad  g \in L^2(V) .
	\ees
        \eL
        
\begin{proof} 
  In a variation of  the proof of Lemma \ref{lem:2}, let  $g \in L^2(V)$, but  this time
  let $\psi$ be the solution of \begin{align} \label{eq:11} \mcT \psi = \sigma^{1/2} g , \end{align} subject to vacuum boundary conditions. Then  set $\phi = \mcP \psi$, implying that $\mcP \psi = \phi = \mcK (\sigma^{1/2} g )$. 
This time, applying $\mcP$ directly to \eqref{eq:11} and recalling that $g$ and $\sigma$ are both independent of $\Omega$, we get
\bes
	\mcP(\Omega\cdot\nabla\psi) + \sigT\phi = \sigThalf  g , 
\ees
and 
so \begin{align*} 
     \sigma^{1/2} \mcK \sigma^{1/2} g \ &= \ \sigma^{1/2} \phi = g - \sigma^{-1/2} \mcP (\Omega \cdot  \nabla \psi).   \end{align*}
   Multiplying each side of this relation by $g$ and integrating over $V$, we get 
\begin{align}\label{eq:12} 
  \langle g,  \mcL  g \rangle \ &= \ \langle g,g \rangle - \int_V \sigma^{-1/2}(\br)  g(\br) \mcP (\Omega \cdot  \nabla \psi(\br, \Omega)) \rd r.  \end{align}
Examining the second term on the right-hand side of  \eqref{eq:12},  we see that this may be written   
\begin{align}\label{eq:13}
  \frac{1}{4 \pi} \int_{\mathbb{S}^2} \int_V  \sigma^{-1/2}(\br)  g(\br) \Omega \cdot  \nabla \psi(\br, \Omega) \rd r \rd \Omega . 
\end{align}
Multiplying  \eqref{eq:11} by $\sigma^{-1}$, we obtain the formula 
$ \sigma^{-1/2} g = \psi + \sigma^{-1} \Omega. \nabla \psi$,
Using this in \eqref{eq:13} and then the divergence theorem again,  we see that \eqref{eq:13} is  
\begin{align*}
  & \frac{1}{4 \pi} \int_{\mathbb{S}^2} \int_V \left(\psi(\br, \Omega) \Omega \cdot  \nabla \psi (\br, \Omega) \rd \br \rd \Omega +
  \sigma^{-1} (\Omega \cdot  \nabla \psi(\br, \Omega) )^2\right)  \rd \br \rd \Omega \\
  & \mbox{\hspace{1cm}} \ \geq\  \frac{1}{8 \pi} \int_{\mathbb{S}^2} \int_V \nabla \cdot  (\psi^2(\br, \Omega) \Omega)\rd \br \rd \Omega \ =\ \frac{1}{8 \pi} \int_{\mathbb{S}^2} \int_{\partial V}   \psi^2(\br, \Omega) \Omega \cdot  \bn (\br) \rd s  \rd \Omega, 
\end{align*}
the inner integral on the right-hand side being over the  surface $\partial V$. This
is non-negative because of the vacuum boundary conditions. Hence \eqref{eq:13} is non-negative and  
combining this   with \eqref{eq:12}, we obtain the result. 
\end{proof}

\bL \label{lem:4} 

	\bes
      \Vert \mcL g \Vert_{L^2(V)}^2 \ = \ \langle {g}, {\mcL^2 g} \rangle \  \leq \ \Vert g \Vert_{L^2(V)}^2 , \quad \text{for all} \quad  g \in L^2(V) .
	\ees
\label{lem:RayQuotRelation}
\eL
\begin{proof}
By \cite[Chapter 104]{RieNag:55}
	the positive-definite self-adjoint operator $\mcL$ possesses a unique positive-definite self-adjoint square root, $\mcL^{1/2}$. 
	Take any $g \in L^2(V)$.  
        Then,  using the self-adjointness of $\mcLhalf$ and
        Lemma \ref{lem:RayQuotBound} (twice),   we obtain
        \begin{align*}\langle g, \mcL^2 g \rangle    \ & =\ \langle g, \mcL^{1/2} \mcL \mcL^{1/2}  g \rangle \ =\ \langle \mcL^{1/2} g,  \mcL \mcL^{1/2}  g \rangle \leq \langle \mcL^{1/2}  g,   \mcL^{1/2}  g \rangle \\
& =\ \langle g,  \mcL   g \rangle \ \leq  \langle g,  g \rangle, 
        \end{align*}       
	as required.
\ep
Using the above results we are now in a position to prove the main theorem.\\

\ \ 




\noindent {\em Proof of Theorem \ref{thm:NormBound}. }
 \ \  For   any $g \in L^2(V, \sigma)$, we have $g \in L^2(V)$ and we can write  $$ \sigT^{1/2} \mcK \sigma^* g \ = \  \mcL  \frac{\sigma^*}{\sigT}
 \sigT^{1/2} g. $$

 Using Lemma \ref{lem:RayQuotRelation}, we then have 
 \begin{align*} \Vert \mcK \sigma^* g \Vert_{L^2(V, \sigma)} \  & =  \Vert \sigma^{1/2} \mcK \sigma^* g \Vert_{L^2(V)} \  = \ \Vert \mcL (\sigma^*/\sigma) \sigma^{1/2}  g \Vert_{L^2(V)} \ \\
    & \leq \
  \Vert (\sigma^* / \sigma) \sigma^{1/2} g \Vert_{L^2(V)} \ \leq \ 
  \Vert \sigma^* / \sigma \Vert_{L^\infty (V)} \Vert g \Vert_{L^2(V, \sigma)}, \end{align*} 
  and the result follows.

  \begin{remark}
    Although we have given the proof here only in the  3D case, the same result holds for classical
    2D and 1D model problems. For example in the 1D ``slab geometry''
    case formulated on the unit interval, the transport equation is
    $$ \mu \frac{\rd \psi }{\rd x}(x, \mu) + \sigma (x) \psi(x, \mu) =  \frac 12 \sigma_S(x)  \int_{-1}^1 \psi(x,\mu') \rd \mu' + Q(x)$$
    where $x \in (0,1)$ and $\mu  \in  (-1,1)$. The counterpart of the  intergral operator $\mcK$ is
    $$ \mcK g (x) = \frac 12 \int_0^1 E_1(\vert \tau(x,y)\vert ) g(y) \rd y , $$
    with $\tau$ denoting the optical path and $E_1$ the exponential integral.
    The counterpart of Theorem \ref{thm:1} for this case
    is proved  using almost identical arguments to those given above (see, e.g., \cite{Bl:16}).  
    \end{remark}

 \subsection{Some applications of Theorem \ref{thm:1} }


 \subsubsection{Convergence of source iteration} 
 From Theorem \ref{thm:1} and Corollary \ref{cor:SI} we immediately have the following result on the convergence of source iteration.
 
\begin{corollary} \label{cor:1} 
	Under the definitions above we have: 	
	\be
	\label{eq:PWconst_errorNormBound}
		\norm{\err^{i+1}}_{L^2(V,\sigma)} \ \leq \ \norm{\dfrac{\sigS}{\sigT}}_{L^\infty(V)}\norm{\err^{i}}_{L^2(V, \sigma)}.
	\ee
        Since, by \eqref{eq:3DCrossSections}, 
       $\norm{\dfrac{\sigS}{\sigT}}_{L^\infty(V)}  < \ 1$, we have $e^i \rightarrow 0$ as $ i \rightarrow \infty$. 	

\end{corollary}

\begin{remark}
 A stronger estimate than \eqref{eq:PWconst_errorNormBound} can be obtained  in the case of constant cross-sections. In \cite{Bl:16} it was shown that on a spatial domain $V$ with diameter $d$  
 $$
 \norm{\err^{i+1}}_{L^2(V)} \ \leq \ \left(\frac{\sigS}{\sigT}\right)  (1 - \exp(-\sigma d)) 
\norm{\err^{i}}_{L^2(V)}.
$$
So on small domains source iteration can still converge rapidly, even if the scattering ratio is close to $1$. 
       \end{remark}

\subsubsection{Data-explicit estimates for the RTE} 
The next corollary gives data-explicit estimates for the pure transport problem and for the RTE.

\begin{corollary}\label{cor:2}  (i) Consider the pure transport problem  $\mcT \psi = g$ with vacuum boundary conditions,   as in Lemma  \ref{lem:1},  and let $\phi$ be the corresponding scalar flux. Then
$$ \Vert \phi \Vert_{L^2(V)} \ \leq \
\frac{1}{\sigma_{\min}} 
\Vert g \Vert_{L^2(V)} .
  $$

(ii)   Consider the RTE  \eqref{eq:3DTEfull} with  vacuum boundary conditions \eqref{eq:3DnonzeroBCs} and let $\phi$ be the corresponding scalar flux. Then
$$ \Vert \phi \Vert_{L^2(V)} \ \leq \
\frac{1}{\sigma_{\mathrm{min}}}
\left(1 - \left\Vert \frac{\sigma_S }{ \sigma} \right\Vert_{L^\infty(V)} \right)^{-1}
\Vert Q \Vert_{L^2(V)} .
  $$

\end{corollary}
\begin{proof}
  (i)   By Theorem \ref{thm:1}, we have $\phi = \mcK g$, so  
  
  $$
  \Vert \phi\Vert_{L^2(V,\sigma)} \ = \ \Vert \sigma^{1/2} \mcK g  \Vert_{L^2(V)} \ = \ \Vert \mcL (\sigma^{-1/2} g) \Vert_{L^2(V)}
 \ \leq \  \Vert \sigma^{-1/2} g \Vert_{L^2(V)}    $$ 
 (using Lemma  \ref{lem:4}),    and this yields the result.

 (ii) By Theorem \ref{thm:1}, the operator $\mcK \sigS$ is a contraction on $L^2(V, \sigma)$ with norm bounded by
 $\Vert \sigS/\sigma \Vert_{L^\infty(V)}< 1$. Hence,  by \eqref{eq:IE} and the Banach Lemma,   we can write 
    $ \phi = (I - \mcK \sigma_S)^{-1} \mcK Q$, with 
    $$ \Vert \phi \Vert_{L^2(V, \sigma)} \ \leq \   \left(1 - \left\Vert \frac{\sigma_S }{ \sigma} \right\Vert_{L^\infty(V)} \right)^{-1} \Vert \mcK Q \Vert_{L^2(V,\sigma)}. $$
    Writing $\sigma^{1/2} \mcK Q = \mcL (\sigma^{-1/2} Q)$ and proceeding as in part (i), we obtain  (ii).  
    \end{proof}

    \begin{remark}  (i) The bounds in Corollary \ref{cor:2} provide the mechanism for estimating the flux $\phi$ in appropriate Bochner norms when the data $\sigma, \sigma_A, \sigma_S$ are random fields, i.e.,  when we wish to quantify how uncertainty in data propagates to uncertainty in the fluxes or in the criticality (see  the next subsection). Integrability in probability space of the right-hand sides in each of the
      estimates  (i) or (ii) above immediately implies the same integrability properties for the resulting flux.  
      In \cite{Pa:18,GrPaSc:18, GrPaSc:18a} this is worked out in detail and the theory
      of multilevel Monte Carlo methods for computing quantities of interest is presented for the RTE in one and two-dimensional models.    

      \end{remark} 
    \subsubsection{Spectral properties of the RTE}

    In the study of nuclear reactor stability, one is concerned with  the eigenvalues $\lambda$ of the
    generalised eigenproblem:
    \begin{align} \label{evp}
      \Omega \cdot \nabla \psi  + \sigma \psi \ = \ \sigma_S \phi + \lambda \sigma_F \phi, 
    \end{align}
    with vacuum boundary condition. Here (in this simplified model problem), 
    $\sigma_F $ is the fission cross-section which is also assumed bounded above and below on $V$ by positive constants,  
    and now  $$\sigma = \sigma_S + \sigma_F + \sigma_A .$$
    In fact one is concerned with the {\em fundamental}  eigenvalue of \eqref{evp}, the smallest in absolute value.
    The reactor is stable and efficient provided the fundamental eigenvalue is close to $1$. In this case the neutrons produced by fission balance the neutrons lost by scattering  and absorption. 

    It is a not completely obvious fact that the spectrum of the problem
    \eqref{evp} is in fact discrete,   real and bounded below by a positive number. This fact can be obtained from the elementary properties which we have derived above.
    \begin{corollary} 
      The eigenvalues of problem \eqref{evp} are    real and  positive.   
    \end{corollary}
    \begin{proof}
      Let $ (\lambda, \psi) $ be an eigenpair of   \eqref{evp}. Then, as in \eqref{eq:IE}, 
      \begin{align} \nonumber  (I - \mcK \sigma_S ) \phi \ = \ \lambda \mcK \sigma_F \phi .   \end{align}
      Multiplying through by $\sigS^{1/2}$ and setting  $v = \sigma_S^{1/2}  \phi$,  we have
      \begin{align} \label{evp1} (I - \mcL_{\sigma_S} ) v \ = \ \lambda \mcL_{\sigma_S} \left(\left(\frac{\sigma_F}{\sigma_S}\right) v\right),  \end{align} 
      where
      $\mcL_{\sigma_S} : = \sigma_S^{1/2} \mcK \sigma_S^{1/2} $.
      Now, since 
      $ \mcL_{\sigma_S} v  =  (\sigma_S/\sigma )^{1/2} \mcL ((\sigma_S/\sigma )^{1/2}   v)$, we can use  
      Lemma \ref{lem:4} to obtain
      $$ \Vert \mcL_{\sigma_S} v\Vert_{L^2(V)} \ \leq \ \left\Vert \frac{\sigma_S}{\sigma} \right\Vert_{L^\infty(V)} \ \Vert v \Vert_{L^2(V)}.     $$ Since $\Vert \sigS  /\sigma\Vert_{L^\infty(V)} < 1  $, $\mcL_{\sigS}$ is  a contraction on $L^2(V)$, and
      $I - \mcL_{\sigma_S}$ is an invertible operator. Thus  $\lambda $ cannot vanish in \eqref{evp1} (since if $\lambda = 0$, then $v = 0$, which implies $\phi = 0$ and hence
      by \eqref{evp}, $\psi = 0$). Hence
      \eqref{evp1} is equivalent to 
\begin{align} \label{evp2} \frac{1}{\lambda} v   \ = \ \mcM  \left(\left(\frac{\sigma_F}{\sigma_S}\right) v\right),  \end{align} 
where $\mcM = (I -  \mcL_{\sigma_S})^{-1} \mcL_{\sigma_S}$.  It is easy to see that $\mcM $ is self-adjoint and compact and has eigenmalues $(1-\mu)^{-1} \mu $ where $\mu$ denotes an  eigenvalue of $\mcL_{\sigma_S}$. Since $\mu$ is always positive and less than $1$ it follows that the eigenvalues of $\mcM$ are all positive and so $\mcM$ is positive definite.      
 Then setting $w = (\sigma_F/\sigma_S)^{1/2} v = \sigma_F^{1/2} \phi$, we have 
\begin{align} \label{evp2} \frac{1}{\lambda} w   \ = \ \mcN  w,  \end{align} 
where 
$$ \mcN = \left(\frac{\sigma_F}{\sigma_S}\right)^{1/2} \mcM \left(\frac{\sigma_F}{\sigma_S}\right)^{1/2}. $$
Since $\mcN$ is also self-adjoint and positive definite, the result follows.

\end{proof}

\begin{remark}
  (i) A more sophisticated argument based on the Krein-Rutman theorem can be used to show that the fundamental eigenvalue is simple with a positive eigenfunction.\\
  (ii) The eigenvalue problem is discussed in detail in the fundamental reference \cite{DaLi:12}. 
  (iii) The quantity $1/\lambda$ is called ``$k-$effective'' in the nuclear engineering literature. 
  \end{remark}

  \bigskip
  
  \noindent {\bf Acknowledgement}\ \ We thank Professor Paul Smith (Wood plc., Poundbury, Dorset, UK) for many helful
  discussions over many years' collaboration, 
  and for supporting the PhD theses of Fynn Scheben, Jack Blake and Matt Parkinson,
  whose work is partially reported here.
  
\bibliographystyle{plain}

\end{document}